\documentclass{sn-jnl}

\usepackage{graphicx}%
\usepackage{multirow}%
\usepackage{amsmath,amssymb,amsfonts}%
\usepackage{amsthm}%
\usepackage{mathrsfs}%
\usepackage[title]{appendix}%
\usepackage{xcolor}%
\usepackage{textcomp}%
\usepackage{manyfoot}%
\usepackage{booktabs}%
\usepackage{algorithm}%
\usepackage{algorithmicx}%
\usepackage{algpseudocode}%
\usepackage{listings}%
\usepackage{verbatim}%
\usepackage{fullpage}

\newtheorem{theorem}{Theorem}[section]% meant for sectionwise numbers
\newtheorem{prop}[theorem]{Proposition}% 
\newtheorem{question}[theorem]{Question}
\newtheorem{lemma}[theorem]{Lemma}
\newtheorem{algo}[theorem]{Algorithm}
\newtheorem{conjecture}[theorem]{Conjecture}

\DeclareMathOperator{\dist}{dist}

\begin{document}

\title{Eternal Distance-$2$ Domination in Trees}

\author*[1]{A Clow} \email{alexander\_clow@sfu.ca}
\author[2]{CM van Bommel}\email{cvanbomm@uoguelph.ca}

\affil[1]{Department of Mathematics, Simon Fraser University, Burnaby, BC, Canada}
\affil[2]{Department of Mathematics and Statistics, University of Guelph, Guelph, ON, Canada}

\abstract{
    We consider the eternal distance-2 domination problem, recently proposed by Cox, Meger, and Messinger, on trees. We show that finding a minimum eternal distance-2 dominating set of a tree is linear time in the order of the graph by providing a fast algorithm. Additionally, we characterize when trees have an eternal distance-2 domination number equal to their domination number or their distance-2 domination number, along with characterizing which trees are eternal distance-2 domination critical. We conclude by providing general upper and lower bounds for the eternal distance-k domination number of a graph, as well as constructing an infinite family of trees which meet said upper bound and another which meets the given lower bound.
}



\maketitle

\section{Introduction}

Consider the problem of stationing ambulances throughout a city.  We desire that each location in the city is close to an ambulance, so that paramedics can respond to a call quickly, but we also need the response time to be maintained after a specific ambulance is assigned to a call.  We can model this problem using the notion of \emph{eternal domination} of a graph, which we outline as follows.  For a graph $G$, we start with an initial set of vertices occupied by ``guards'' (or in our application, ambulances) that form a \emph{dominating set} of the graph, that is, every vertex not in the set is adjacent to a vertex in the set.  Then in a series of steps, a vertex is chosen (``attacked'', requires an ambulance), and we require a new set of vertices that include the chosen vertex and for which there is a matching between the previous set and the new set where pairs of vertices are identical or adjacent (intuitively, each guard or ambulance moves to an adjacent location or stays put).  We say that a collection of sets $\{D_1, D_2, \ldots, D_m\}$ eternally dominates $G$ if for any sequence of attacks $A_1, A_2, \ldots$, there is a sequence of sets $D_{i_1}, D_{i_2}, \ldots$ such that $D_{i_j}$ contains the attacked vertex $A_j$, and there is a matching between the vertices of $D_{i_j}$ and $D_{i_{j + 1}}$ such that pairs of vertices are identical or adjacent.  The minimum number of guards required, that is the minimum cardinality of an eternal dominating set, is called the \emph{eternal domination number} of $G$, and is denoted $\gamma_{all}^\infty(G)$.  Here, the subscript all refers to the fact that each guard is able to move in response to an attack, as introduced by Goddard, Hedetniemi, and Hedetniemi~\cite{goddard2005eternal}, in contrast to the version originally introduced by Burger et al.~\cite{burger2004infinite} where only one guard is permitted to move in response to an attack.  

Several works have studied the eternal domination number in recent years, as surveyed by Kloystermeyer and Mynhardt~\cite{klostermeyer2016protecting}.  Goddard, Hedetniemi, and Hedetniemi~\cite{goddard2005eternal} established the domination number as a fundamental lower bound and the clique-star cover number as a fundamental upper bound, established exact values for cliques, complete bipartite graphs, paths, cycles, and Cayley graphs, and determined additional upper bounds in terms of the 2-domination number, independence number, and clique-connected cover number.  Klostermyer and MacGillivray~\cite{klostermeyer2009eternal} provide a linear-time algorithm for determining the eternal domination number of trees.  In a subsequent paper, Klostermeyer and MacGillivray~\cite{klostermeyer2014eternal} characterize trees achieving various upper and lower bounds on the eternal domination number.

The eternal domination number of grids has been extensively studied.  Beaton, Finbow, and MacDonald~\cite{beaton2013eternal} considered the eternal domination number of $4 \times n$ grids. Finbow, Messinger, and van Bommel~\cite{finbow2015eternal} establish upper and lower bounds on the eternal domination number of $3 \times n$ grids.  Bounds on the eternal domination number of $5 \times n$ grids were established by van Bommel and van Bommel~\cite{van2016eternal}.  Lamprou, Martin, and Schewe~\cite{lamprou2019eternally} provided a general upper bound for grid graphs that is optimal asymptotically.  Bounds for strong grid graphs were studied by McInerney, Nisse, and Pérennes~\cite{mc2019eternal} and Gagnon et al.~\cite{gagnon2020method}.

Henning, Kloystermeyer, and MacGillivray~\cite{henning2017bounds} demonstrated a tight upper bound for connected graphs with minimum degree 2, and improved the bound for the class of cubic bipartite graphs.  Cohen et al.~\cite{cohen2018spy} study a generalization called the spy-game, which they show to be NP-hard.  Finbow et al.~\cite{finbow2018note} demonstrated the advantage of allowing multiple guards to occupy the same vertex, and give exponential-time algorithms for calculating eternal domination numbers.  In this work we will allow guards to occupy the same vertex.  Bla{\v{z}}ej, K{\v{r}}i{\v{s}}t’a, and Tom{\'a}{\v{s}}~\cite{blavzej2019m} determine bounds and a linear time algorithm for eternal domination numbers of cactus graphs.

Here, we assume an ambulance is not limited to moving to an adjacent location, rather they are allowed to move to locations up to distance $k$ away. Then we can model this relaxed version of the problem using the notion of \emph{eternal distance-$k$ domination} of a graph, as introduced by Cox, Meger, and Messinger~\cite{cox2023eternal}, defined as follows.  For a graph $G$, we start with an initial set of vertices occupied by ``guards'' (or in our application, ambulances) that form a \emph{distance-$k$ dominating set} of the graph, that is, every vertex not in the set is at distance at most $k$ to a vertex in the set.  Then in a series of steps, a vertex is chosen (``attacked'', requires an ambulance), and we require a new set of vertices that include the chosen vertex and for which there is a matching between the previous set and the new set where pairs of vertices are within distance $k$ (intuitively, an ambulance responds to the call, and the remaining ambulances are redistributed across the city to maintain their response time.  We say that a collection of sets $\{D_1, D_2, \ldots, D_m\}$ eternally dominates $G$ if for any sequence of attacks $A_1, A_2, \ldots$, there is a sequence of sets $D_{i_1}, D_{i_2}, \ldots$ such that $D_{i_j}$ contains the attacked vertex $A_j$, and there is a matching between the vertices of $D_{i_j}$ and $D_{i_{j + 1}}$ such that pairs of vertices are within distance $k$.  The minimum number of guards required, that is the minimum cardinality of an eternal distance-$k$ dominating set is called the \emph{eternal distance-$k$ dominaton number} of $G$, and is denoted $\gamma_{all, k}^\infty(G)$.  Consider Figure~\ref{fig:Def Example} for two graphs whose domination, distance-$2$ domination, eternal domination, and eternal distance-$2$ domination numbers are given. Cox, Meger, and Messinger~\cite{cox2023eternal} considered general bounds on, the computational complexity of computing, and the exact values for small classes, of $\gamma_{all,k}^\infty$. Additionally, they developed reductions for trees and present the following open problems:

\begin{figure}
    \centering
    \includegraphics[scale = 0.125]{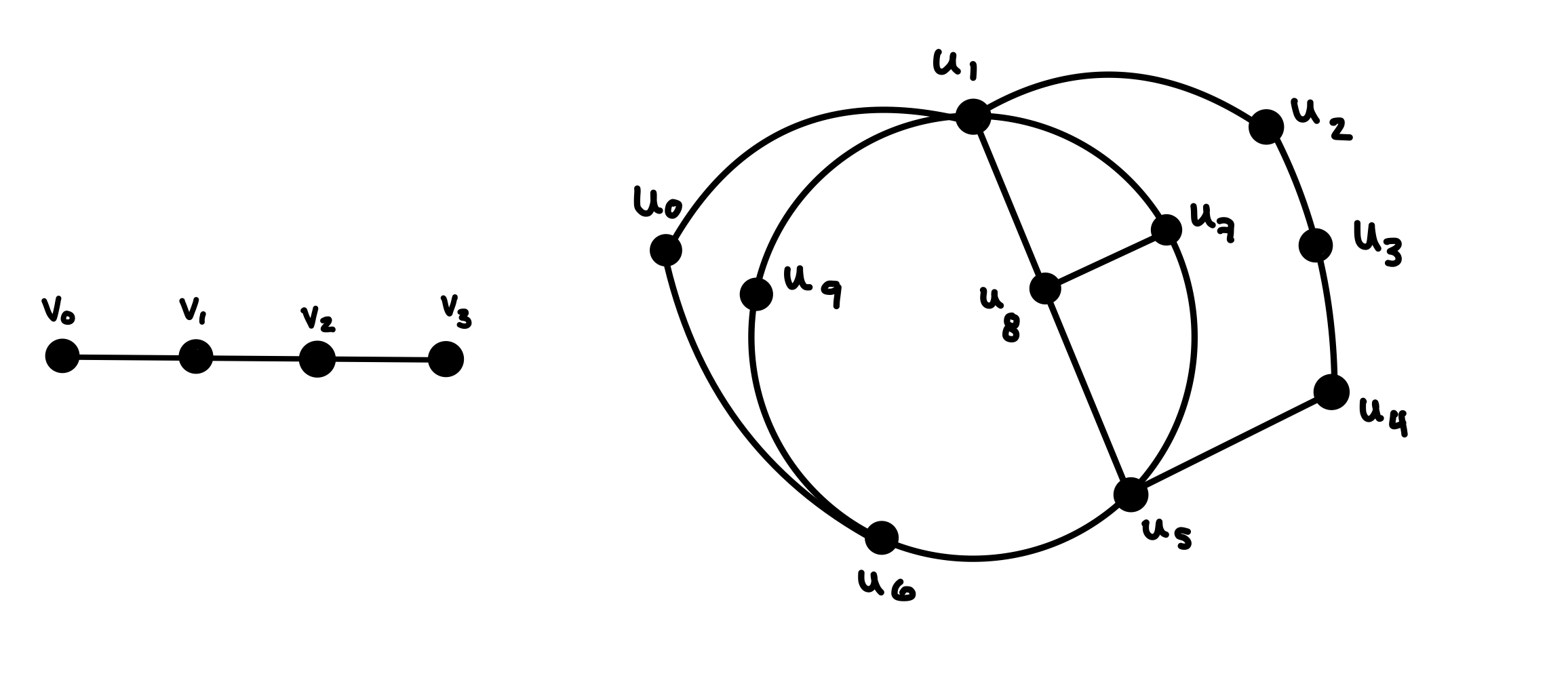}
    \caption{Two graphs $G$ (left) and $H$ (right) such that $\gamma(G) = 2 = \gamma_{all,1}^\infty(G)$, $\gamma_2(G) = 1 < \gamma_{all,2}^\infty(G) = 2$, $\gamma(H) = 3 < \gamma_{all,1}^\infty(H) = 4$, and $\gamma_2(H) = 2 = \gamma_{all,2}^\infty(H)$.
    }
    \label{fig:Def Example}
\end{figure}

\begin{question} \label{q-min}
    For what class of graphs, $\mathcal{G}$, is $\gamma_k(G) = \gamma_{all, k}^\infty(G)$ for all $G \in \mathcal{G}$?
\end{question}

\begin{question} \label{q-extreme}
    Let $\mathcal{G}_{n, m}$ be the family of simple graphs on $n$ vertices and $m$ edges.  For a fixed $n$ and $m$, what are the graphs with the smallest eternal distance-$k$ domination number, or largest eternal distance-$k$ domination number?
\end{question}

\begin{question} \label{q-possible}
    Given a fixed $n$ and fixed $k$, what possible values can $\gamma_{all, k}^\infty$ take on, for graphs $G$ of order $n$?
\end{question}

\begin{question} \label{q-max}
    Which graphs $G$ have the property that $\gamma_{all, 2}^\infty(G) = \gamma(G)$?  Can we characterize the trees with this property?
\end{question}

\begin{question} \label{q-pn}
    Suppose that for every minimum dominating set of a tree $T$, each vertex in the dominating set has at least two private neighbours.  Then is $\gamma_{all, 2}^\infty(T) = \gamma(T)$?
\end{question}

In this work, we focus primarily on the case of $k = 2$ in trees.  We develop sufficient reductions to provide an algorithm to determine the eternal distance-2 domination number of trees in linear time.  Our main approach is to consider the low-diameter subgraphs that are most leaf-like; we develop this idea formally in Section~3.  Using these reductions, we provide a linear time algorithm for the computation of the eternal distance-2 domination number of trees.  We then provide characterizations of trees for which $\gamma_{all, 2}^\infty = \gamma$, addressing Question~\ref{q-max} for the class of trees and allowing us to answer Question~\ref{q-pn} in the negative. Additionally, we characterize the trees that are eternal distance-$2$ domination critical, as well as the trees for which $\gamma_{all, 2}^\infty = \gamma_2$, partially addressing Question~\ref{q-min}.  Finally, we present extremal families of trees for eternal distance-$k$ domination, offering a solution to Question~\ref{q-extreme} in the case $m=n-1$.

\section{Preliminaries}

We provide the following results relating domination parameters that will be used throughout.  We begin with the following bound on the eternal distance-$k$ domination number in terms of distance domination, as observed by Cox, Meger, and Messinger~\cite{cox2023eternal}.

\begin{prop} \cite{cox2023eternal}
    For any graph $G$ and integer $k \ge 2$,
    \[ \gamma_k(G) \le \gamma_{all, k}^\infty(G) \le \gamma_{\left \lfloor \frac{k}{2} \right \rfloor}(G).\]
\end{prop}

Equality between $\gamma(G)$ and $\gamma_2(G)$ for trees was characterized by Raczek~\cite{raczek2011graphs}, which forces equality of the eternal distance-2 domination number.  We state the characterization here and consider the question of equality of the bounds in later sections.  Let $\mathbb{T}$ be the family of all trees $T$ that can be obtained from sequence $T_1, \ldots, T_j$ ($j \ge 1)$ of trees such that $T_1$ is the path $P_2$ and $T = T_j$, such that $T_{i+1}$ can be obtained recursively from $T_{i}$ by the operation $\mathbb{T}_1$, $\mathbb{T}_2$, or $\mathbb{T}_3$:
\begin{itemize}
    \item {\bf Operation} $\mathbb{T}_1$.  The tree $T_{i + 1}$ is obtained from $T_i$ by adding a vertex $x_1$ and the edge $x_1 y$ where $y \in V(T_i)$ is a stem vertex of $T_i$.
    \item {\bf Operation} $\mathbb{T}_2$.  The tree $T_{i + 1}$ is obtained from $T_i$ by adding a path $(x_1, x_2, x_3)$ and the edge $x_1 y$ where $y \in V(T_i)$ is neither a leaf nor a stem vertex in $T_i$.
    \item {\bf Operation} $\mathbb{T}_3$.  The tree $T_{i + 1}$ is obtained from $T_i$ by adding a path $(x_1, x_2, x_3, x_4)$ and the edge $x_1 y$ where $y \in V(T_i)$ is a stem vertex in $T_i$.
\end{itemize}
Additionally, let $P_1$ belong to $\mathbb{T}$.  Then the following is established.

\begin{theorem}[\cite{raczek2011graphs}]\label{lemma: gamma=gamma_2}
    Let $T$ be a tree. Then $T \in \mathbb{T}$ if and only if $\gamma(T) = \gamma_2(T)$.
\end{theorem}

We next note the exact value for the eternal distance-$k$ domination number has been calculated for paths by Cox, Meger, and Messinger~\cite{cox2023eternal}. In Section~7 we will see that paths are an extremal family of graphs in terms of eternal distance-$k$ domination.

\begin{theorem} \cite{cox2023eternal}
    For $n \ge 1$ and $k \ge 1$, $\gamma_{all, 2}^\infty(P_n) = \left \lceil \frac{n}{k + 1} \right \rceil$.
\end{theorem}

Finally, to characterize the trees that are eternal distance-2 domination critical, we use the concept of the 1-sum of two graphs to build the family.  For graphs $G$ and $H$, the \emph{1-sum} of $G$ and $H$ at vertex $u$ of $G$ and $v$ of $H$ is the graph formed by taking the disjoint union of graphs $G$ and $H$ and identifying vertices $u$ and $v$.

\section{Reductions \& Complexity for Eternal Distance-$2$ Domination in Trees}

Let $T = (V,E)$ be a tree. We say $v \in V$ is a $0$-leaf iff $v$ is a leaf, and for $k>0$, we say $v$ is a $k$-leaf iff $v$ is adjacent to a $(k-1)$-leaf and all but perhaps one of the neighbours of $v$ are $t$-leaves, where $t<k$. For a $1$-leaf, $u$, let $L(u)$ be the set of all leaves adjacent to $u$, and let $L[u] = L(u) \cup \{u\}$. Given a $k$-leaf, $v \in V$, $k>1$, let $L(v) = \cup_{u \in N(v)} L[u]$, where $u$ is a $t$-leaf and $t<k$, and let $L[v] = L[v] \cup \{v\}$.  We begin by demonstrating that a tree with a sufficiently large diameter contains a $k$-leaf.

\begin{lemma}\label{Lemma: Big Diameter Implies k-leaf}
    Let $T = (V,E)$ be a tree. If the diameter of $T$ is at least $2k$, then there exists a $v \in V$ which is a $k$-leaf.
\end{lemma}

\begin{proof}
Let $T = (V,E)$ be a tree with diameter at least $2k$.  Let $P = v_0 v_1 v_2 \ldots v_m$ be a longest path in $T$.  If $v_k$ is not a $k$-leaf, then there exists a path $Q = w_0 w_1 w_2 \ldots w_t x_k$, where $t \ge k$.  But then if we replace the vertices $v_0 \ldots v_k$ of $P$ with $Q$, we obtain a longer path in $T$, a contradiction.  Hence, $T$ contains a $k$-leaf.
\end{proof}

Next, we demonstrate a reduction strategy for computing the eternal distance-2 domination number based on 2-leaves.

\begin{lemma}\label{Lemma: k-leaf Reduction}
    Let $T$ be a tree with diameter at least $4$. If $v$ is a $2$-leaf in $T$ such that $T[L(v)]$ has diameter $2$, let $T' = T - L[v]$, otherwise, let $T' = T - L(v)$. Then, $\gamma_{all,2}^\infty(T) = \gamma_{all,2}^\infty(T') + 1$.
\end{lemma}

\begin{proof}
Let $T = (V,E)$ be a tree with diameter at least $2k$ and let $v \in V$ be a $2$-leaf. Let $T'$ be defined as in the statement of the Lemma. We will show $\gamma_{all,2}^\infty(T) = \gamma_{all,2}^\infty(T') + 1$ by demonstrating that $\gamma_{all,2}^\infty(T) \leq  \gamma_{all,2}^\infty(T') + 1$ and $\gamma_{all,2}^\infty(T) \geq \gamma_{all,2}^\infty(T') + 1$.

To begin observe that as $v$ is a $2$-leaf and $T$ is a tree, there must be at least $1$ guard in $L[v]$ in every distance-$2$ dominating set of $T$. Otherwise, there exists a vertex in $L(v)$ which is not distance-$2$ dominated. Hence, for every eternally distance-$2$ dominating set there is at least $1$ guard in $L[v]$. 

\vspace{0.25cm}
\underline{Case.1:} $T[L(v)]$ has diameter $2$. Then for all vertices $u \in L[v]$, a guard at $u$ is within distance $2$ of every other vertex in $ L[v]$. Hence, the one guard which must be in $L[v]$ can eternally distance-$2$ dominate $L[v]$. This implies that $\gamma_{all,2}^\infty(T) \leq  \gamma_{all,2}^\infty(T') + 1$, as one extra guard is sufficient to guard $L[v]$, while $\gamma_{all,2}^\infty(T')$ is sufficient to guard the rest of the graph.

Note that $\gamma_{all,2}^\infty(T) \geq  \gamma_{all,2}^\infty(T') + 1$ follows directly from the fact there must always be a guard in $ L[v]$. This is because if the guard $g_1$ in $ L[v]$ were to ever move to protect a vertex $x \in V\setminus L[v]$, then another guard, $g_2$, would have to enter $L[v]$ to remain a distance-$k$ dominating set. Given $T$ is a tree, this is never advantageous for the guards, as if $g_2$ is within distance $2$ of $L[v]$ while not being in $L[v]$, this implies $g_2$ is also within distance $2$ of $x$, so $g_2$ can move directly to $x$ and $g_1$ can remain in $L[v]$.

\vspace{0.25cm}
\underline{Case.2:} $T[L(v)]$ has diameter at least $3$. Then we will demonstrate $\gamma_{all,2}^\infty(T) \leq  \gamma_{all,2}^\infty(T') + 1$ by providing a winning strategy with for the guards using exactly $\gamma_{all,2}^\infty(T') + 1$ guards. Place $\gamma_{all,2}^\infty(T')$ guards on $V \setminus L(v)$ and let them proceed as if playing on $T'$, next place an extra guard, $g_v$, on $v$. If the attacker attacks vertices in $V \setminus L(v)$, then let the guards $\gamma_{all,2}^\infty(T')$ proceed as if playing on $T'$, while $g_v$ does not move. 

If a vertex in $L(v)$ is attacked let $g_v$ defend against the attack, while the other $\gamma_{all,2}^\infty(T')$ guards respond as if $v$ was attacked in $T'$. This will place a guard $g$ on $v$. If the next attack is also on $L(v)$, then let $g$ move to respond to the attack, while $g_v$ returns to $v$ and the guards in $V \setminus L[v]$ do not move. As long as the attacker attacks vertices in $L(v)$, guards $g$ and $g_v$ continue this strategy where the guard at $v$ responds, while the other moves to $v$. Should the attacker attack a vertex outside of $L(v)$, then the guard currently at $v$, say $g$ without loss of generality, assumes the role of one of the $\gamma_{all,2}^\infty(T')$ guards protecting $T'$, while the other guard (again without loss of generality) $g_v$, returns to $v$. As this returns the game to its initial state, $\gamma_{all,2}^\infty(T) \leq  \gamma_{all,2}^\infty(T') + 1$.

The fact that $\gamma_{all,2}^\infty(T) \geq  \gamma_{all,2}^\infty(T') + 1$ follows by a similar argument as case 1, where it is never advantageous for the guard $g_v$ whose job it is to protect $L[v]$ to protect a vertex $x$ outside of $L[v]$, as any guard $g$ which begins outside of $L[v]$ and would take on the role of defending $L[v]$ could simply protect $x$ directly.

This completes the proof as we have shown that $\gamma_{all,2}^\infty(T) = \gamma_{all,2}^\infty(T') + 1$ as desired in both cases.
\end{proof}

Note that the approach used in Lemma~\ref{Lemma: k-leaf Reduction} does not easily extend for $\gamma_{all,k}^\infty$ when $k>2$. For an example of this see Figure~\ref{fig:k>2 problem}. Note that when $k>2$, we may conclude that $\gamma_{all,k}^\infty(T) \leq \gamma_{all,k}^\infty(T') +1$, however we cannot guarantee that this bound reaches equality, because the guard $g_2$ (see proof of case 1) might have distance greater than $k$ to $x$. This is because supposing $g_1$ starts at $v$ and $g_2$ starts at $z$, $\dist(v, x) \leq k$ implies that the neighbour of $v$ not in $L[v]$, call it $y$, has $\dist(x,y) \leq k-1$, while, $\dist(z,y) \leq k-1$. Hence, $\dist(z,x) \leq 2k-2$ is an upper bound that cannot be improved. For $k \leq 2$ this is fine, as $k \geq 2k-2$, however for $k>2$, $k< 2k-2$ implies $g_2$ may be unable to reach $x$ in a single move. Notice that Figure~\ref{fig:k>2 problem} can be generalized to all cases $k \geq 3$ by appending path of length $k-3$ to each leaf in $T$ and $T'$ respectively.

It is also significant to point out that the trees $T$ and $T'$ in Figure~\ref{fig:k>2 problem} are counterexamples to Proposition~3 from \cite{cox2023eternal}. That is, in an identical way to how Lemma~\ref{Lemma: k-leaf Reduction} does not generalize to the $k>2$ case, Proposition~3 from \cite{cox2023eternal} will not apply to $k>2$. 

When the diameter of the tree is small, the eternal distance-2 domination number can be directly computed as follows.

\begin{figure}
    \centering
    \includegraphics[scale = 0.5]{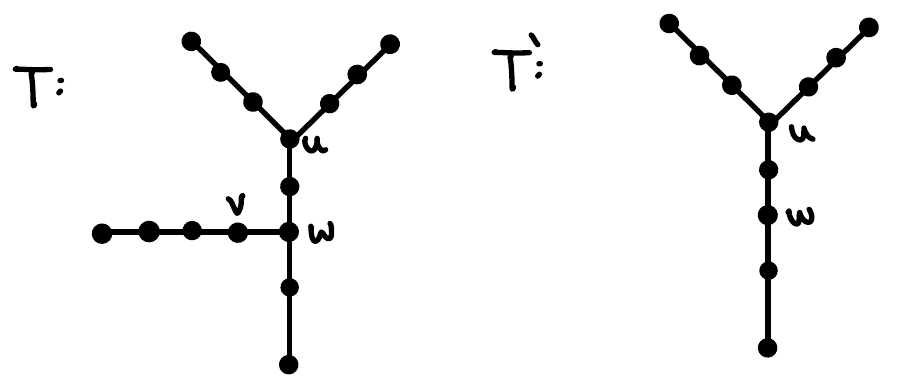}
    \caption{A tree $T$ where, $\gamma_{all,3}^\infty(T) \neq \gamma_{all,3}^\infty(T') + 1$. In this case both trees have eternal distance-$3$ domination number $3$.}
    \label{fig:k>2 problem}
\end{figure}

\begin{lemma}\label{Lemma: Small Diameter Characterization}
    Let $T = (V,E)$ be a tree and let $k>0$ be an integer and let $d$ be the diameter of $T$,
    \begin{itemize}
        \item if $k < d < 2k$, then $\gamma_{all,k}^\infty(T) = 2$, and
        \item if $d \leq k$, then $\gamma_{all,k}^\infty(T) = 1$.
    \end{itemize}
\end{lemma}

\begin{proof}
Recall that \cite{cox2023eternal} showed that for all graphs $G$ and integers $k$, $\gamma_{all,k}^\infty (G) = \gamma_{all,1}^\infty (G^k)$. Suppose  $k < d < 2k$. Then $T^k$ has a universal vertex but is not a complete graph. This implies $\gamma_{all,1}^\infty (T^k) = 2$, which implies $\gamma_{all,k}^\infty (T) = 2$. Now suppose, $d \leq k$, then $T^k$ is complete, implying $\gamma_{all,1}^\infty (T^k) = \gamma_{all,k}^\infty (T) = 1.$
\end{proof}

We conclude this section with an algorithm, based on the previous results, that will allow us to compute the eternal distance-2 domination number in linear time.

\begin{algo} \label{alg: edtd}
    Let $T$ be a tree rooted at a vertex $r$.  We compute the eternal distance-2 domination number of $T$ as follows:
    \begin{enumerate}
        \item Set $\gamma = 0$ and let $S = \{r\}$ be a stack.
        \item While the stack is not empty, let $x$ be the top vertex of the stack.
        \begin{enumerate}
            \item If the subtree of $T$ consisting of $x$ and its descendants has depth at least 3, add each child of $x$ whose subtree has depth at least two to the stack.
            \item If the subtree of $T$ consisting of $x$ and its descendants has depth 2, remove $x$ from the stack, increment $\gamma$, and replace $T$ with $T - D[x]$ if $x$ has one child, and $T - D(x)$ otherwise, where $D(x)$ is the set of descendants of $x$.
            \item Otherwise, remove $x$ from the stack.
        \end{enumerate}
        \item If $T$ is not empty, increment $\gamma$.
        \item Return $\gamma$.
    \end{enumerate}
\end{algo}

\begin{theorem}\label{Thm: complexity}
    Let $T$ be a tree.  Then the output, $\gamma$, of Algorithm~\ref{alg: edtd}, for an arbitrary vertex $r$ of $T$, is equal to  $\gamma_{all, 2}^\infty(T)$.  Moreover, the running time of Algorithm~\ref{alg: edtd} for such an input is linear in the number of vertices.
\end{theorem}

\begin{proof}
    We proceed by induction on the number of vertices in the following way, for any rooted tree $T$ of order less than $n$ with root $r$, Algorithm~\ref{alg: edtd} will output $\gamma_{all,k}^\infty(T)$. 
    
    As our base case consider a tree $T =(V,E)$ rooted at a fixed but arbitrary vertex $r$ which has depth at most 2. Then $T$ has diameter at most 4.  Note in particular that this is the case for all graphs on at most 3 vertices.  Suppose $T$ has diameter at most 2, then by Lemma~\ref{Lemma: Small Diameter Characterization}, $\gamma_{all, 2}^\infty(T) = 1$.  If $r$ is a universal vertex of $T$, then $T$ has depth 1, so the algorithm returns 1. Otherwise, $r$ is a leaf and $T$ has depth $2$, so $L[x] = V$, hence, $T$ is replaced by $\emptyset$, and the algorithm returns 1.  If $T$ has diameter 3 or 4, then by Lemma~\ref{Lemma: Small Diameter Characterization}, $\gamma_{all, 2}^\infty(T) = 2$.  We have that $T$ has depth 2 and $r$ has multiple children. Thus $T$ is replaced by the single vertex $r$, so the algorithm returns 2.
    
    Now let $T$ be a tree rooted at $r$ such that $T$ has depth at least 3.  Let $v$ be the first vertex to appear on top of the stack for which the subtree of $T$ consisting of $v$ and its descendants has depth 2; clearly $v \neq r$, furthermore, it should be clear that a vertex $v$ will appear at some point during the run-time of Algorithm~\ref{alg: edtd} (as there must exist a descendant of $r$ with the properties of $v$ given $T$ has depth at least $3$, while vertices on top of the stack of type (a) add vertices to the stack whose subtree has smaller depth, vertices of type (b) are exactly vertices $v$, while vertices of type (c) are removed from the stack not to be considered again). Then $v$ is a 2-leaf of $T$, and the diameter of $T$ is at least 3. If $T$ has diameter at most 4, then by Lemma~\ref{Lemma: Small Diameter Characterization}, $\gamma_{all, 2}^\infty(T) = 2$.  Then $T$ is replaced with a nonempty tree of diameter at most 2.  So the algorithm returns 2 as desired by induction, noting that the steps of the algorithm applied to this tree are a subset of the steps applied to $T$.
    
    If $T$ has diameter greater than 4, then the ancestor of $v$ is not in $L[v]$.  If $\deg(v) = 2$, then by Lemma~\ref{Lemma: k-leaf Reduction}, $\gamma_{all, 2}^\infty(T) = \gamma_{all, 2}^\infty(T - L[v]) + 1$, which by the induction hypothesis is precisely the value returned by the algorithm.  If $\deg(v) \ge 3$,     then by Lemma~\ref{Lemma: k-leaf Reduction}, $\gamma_{all, 2}^\infty(T) = \gamma_{all, 2}^\infty(T - L(v)) + 1$, which by the induction hypothesis is precisely the value returned by the algorithm.

    Finally, we observe that each vertex of the tree is only kept on the stack if its corresponding subtree has depth at least 3. When we return to this vertex, its descendants have depth at most 1, so it has depth at most 2, and is therefore removed from the stack. Furthermore, a vertex is considered while not on top of the stack only when first, second, or third ancestor is on top of the stack during the while loop. Thus, each vertex is considered at most a constant number of times and each consideration is linear time. Therefore, the algorithm runs in linear time. This completes the proof.
\end{proof}

\section{Trees with $\gamma_{all,2}^\infty = \gamma$}

In this section we resolve the question posed in \cite{cox2023eternal} of characterizing when a tree has eternal distance-$2$ domination number equal to its domination number. The primary tool in our analysis is the reduction given in Lemma~\ref{Lemma: k-leaf Reduction}. We begin this section by pointing out that there are trees $T$ where every vertex in a minimum dominating set has at least $2$ private neighbours, but the eternal distance-$2$ domination number is arbitrarily far from the domination number. This resolves another question in \cite{cox2023eternal} in the negative.

For an example of such a tree take a star $K_{1,n}$ and add two leaves to each leaf of $K_{1,n}$ to form $T_n$. Then, $\gamma(T_n) = n$ while every vertex in the unique minimum dominating set in this tree has exactly two private neighbours. However as the diameter of $T_n$ is $4$, Lemma~\ref{Lemma: Small Diameter Characterization} implies $\gamma_{all,2}^\infty(T_n) = 2$. 

We first note the following lemma which is a special case of Proposition 1 from ~\cite{cox2023eternal}.

\begin{lemma}[\cite{cox2023eternal}]\label{Lemma: Domination >= Eternal 2-Domination}
Let $G = (V,E)$ be a graph. Then, $\gamma_{all,2}^\infty(G) \leq \gamma(G)$.
\end{lemma}

In order to characterize the trees whose domination number is equal to their eternal distance-$2$ domination number we must first define the following family of trees and explore several properties of this family. We define the family of trees $\mathcal{T}$ recursively as follows:
\begin{enumerate}
    \item Every tree with diameter at most 3 is in $\mathcal{T}$, and
    \item If $T \in \mathcal{T}$, then the tree $T'$ formed by appending a star $K_{1, m}$ with $m \ge 2$ with an edge from a leaf of $K_{1, m}$ to any vertex of $T$ is also in $\mathcal{T}$, and
    \item If $T \in \mathcal{T}$ and $v \in V(T)$ is a leaf such that there is a minimum dominating set containing $v$, then for all trees $T''$ formed by appending a star $K_{1,m}$ where $m \geq 1$ to $v$ by adding an edge from $v$ to the high degree vertex of the star, as well as appending $t\geq 1$  leaves to $v$, is also in $\mathcal{T}$.
    \item If $T \in \mathcal{T}$ and $v \in V(T)$ is a leaf such that $\gamma(T-v)=\gamma(T) - 1$, then $T'''$ formed by appending two stars $K_{1,m}$ and $K_{1,M}$ where $m,M\geq 1$ to $v$ by adding an edge from $v$ to the high degree vertex of the stars, is also in $\mathcal{T}$.
\end{enumerate}

\begin{theorem}\label{Thm: gamma = gamma_(all,2)}
    If $T$ is a tree, then $\gamma_{all, 2}^\infty(T) = \gamma(T)$ if and only if $T \in \mathcal{T}$.
\end{theorem}

\begin{proof}
    Suppose $T \in \mathcal{T}$.  If $T$ has diameter at most 2, then $T$ has a universal vertex, so $\gamma_{all, 2}^\infty(T) = \gamma(T) = 1$.  If $T$ has diameter 3, then $\gamma(T) = 2$ since $T$ does not have a universal vertex, but the two non-leaves form a dominating set, and $\gamma_{all, 2}^\infty(T) = 2$ by Lemma~\ref{Lemma: Small Diameter Characterization}.

    Suppose all trees $S \in \mathcal{T}$ on fewer than $n$ vertices satisfy $\gamma_{all, 2}^\infty(S) = \gamma(S)$, and let $T \in \mathcal{T}$ be a tree with $n$ vertices and diameter at least 4.  Suppose $T$ is formed by appending $K_{1, m}$, $m\geq 2$, to some $T' \in \mathcal{T}$.  Then we have $\gamma_{all, 2}^\infty(T) = \gamma_{all, 2}^\infty(T') + 1$ by Lemma~\ref{Lemma: k-leaf Reduction}.  Moreover, $\gamma(T) \le \gamma(T') + \gamma(K_{1, m}) = \gamma(T') + 1$, and any minimum dominating set of $T$ contains either the stem of $K_{1, m}$ or its unique leaf neighbour, which dominates no vertex of $T'$, so $\gamma(T') \le \gamma(T) - 1$.  Therefore, $\gamma(T) = \gamma(T') + 1$.  Since $\gamma_{all, 2}^\infty(T') = \gamma(T')$ by the induction hypothesis, it follows that $\gamma_{all, 2}^\infty(T) = \gamma(T)$.

    Now suppose $T$ is formed from some $T'' \in \mathcal{T}$ by taking a leaf $v$ for which there is a minimum dominating set of $T''$ containing $v$, and appending to $v$ a star $K_{1,m}$ where $m \geq 1$ by adding an edge to the high degree vertex of the star, as well as appending $t\geq 1$  leaves. Then $v \in V(T)$ is a $2$-leaf where $T[L[v]]$ has diameter $3$ and $T'' = T - L(v)$. Then Lemma~\ref{Lemma: k-leaf Reduction} implies $\gamma_{all,2}^\infty(T) = \gamma_{all,2}^\infty(T'')+1$. Furthermore, let $D\subset V(T'')$ be a minimum dominating set of $T''$ where $v \in D$. Then, $D \cup \{u\}$ where $u$ is the high degree vertex of the star appended to $v$ in $T$ is a dominating set of $T$. This implies $\gamma(T)\leq \gamma(T'')+1$. As $T'' \in \mathcal{T}$ implies $\gamma_{all,2}^\infty(T'') = \gamma(T'')$ and Lemma~\ref{Lemma: Domination >= Eternal 2-Domination} implies $  \gamma(T'') +1 = \gamma_{all,2}^\infty(T'') +1 = \gamma_{all,2}^\infty(T)  \leq \gamma(T)$ we conclude that $\gamma(T)= \gamma(T'')+1$. Hence, $\gamma_{all,2}^\infty(T) = \gamma(T)$ as required.

    Finally, suppose $T$ is formed by appending two stars $K_{1,m}$ and $K_{1,M}$ where $m,M \geq 1$ to $v$ by adding an edge from $v$ to the the high degree vertices of each star, where $v \in V(T''')$ for some $T''' \in \mathcal{T}$ and leaf $v$ satisfying $\gamma(T'''-v) = \gamma(T''')-1$. Then $v$ is a $2$-leaf in $T$ and Lemma~\ref{Lemma: k-leaf Reduction} implies $\gamma_{all,2}^\infty(T) = \gamma_{all,2}^\infty(T''')+1$. As $\gamma(T'''-v) = \gamma(T''')-1$, all minimum dominating sets of $T'''-v$ do not contain $v$ or a neighbour of $v$. Then, letting $D$ be a minimum dominating set of $T'''-v$ and letting $u_1$ and $u_2$ be the high degree vertices of the appended stars we see that $D\cup \{u_1,u_2\}$ is a dominating set in $T$, hence, $\gamma(T) \leq \gamma(T'''-v)+2=\gamma(T''')+1$.
    It is easy to see that $\gamma(T) \geq \gamma(T'''-v)+2=\gamma(T''')+1$ as $T$ contains $2$ leaves with no common neighbours adjacent to no vertices in $T'''-v$. Thus, $\gamma(T) = \gamma(T''')+1$ implying $\gamma(T) = \gamma_{all,2}^\infty(T)$ as required.

    Conversely, suppose there exists a tree $X$ such that $X \notin \mathcal{T}$, but $\gamma_{all, 2}^\infty(X) = \gamma(X)$.  Let $Y$ be a minimal such tree with respect to induced subgraph. Then the diameter of $Y$ is at least $4$.  It follows from  Lemma~\ref{Lemma: Big Diameter Implies k-leaf} that there exists a $2$-leaf in $Y$.  
    
    Suppose $v \in V(Y)$ is a $2$-leaf such that $Y[L[v]]$ has diameter $2$, and let $Y' = Y - L[v]$. By Lemma~\ref{Lemma: k-leaf Reduction}, we have $\gamma_{all, 2}^\infty(Y) = \gamma_{all, 2}^\infty(Y') + 1$, and by Lemma~\ref{Lemma: Domination >= Eternal 2-Domination}, we have $\gamma_{all, 2}^\infty(Y') \le \gamma(Y')$.  Consider a leaf $\ell \neq v$ in $Y[L[v]]$.  Then any dominating set of $Y$ contains a vertex in $L[v]$ in order to guard $\ell$, and that guards no vertex in $Y'$.  Thus $\gamma(Y') < \gamma(Y)$. But this implies $\gamma(Y)-1 = \gamma_{all,2}^\infty(Y)-1 = \gamma_{all,2}^\infty(Y') \leq \gamma(Y') \leq \gamma(Y)-1$. It follows that $\gamma_{all, 2}^\infty(Y') = \gamma(Y')$.  By the minimality of $Y$, we have $Y' \in \mathcal{T}$.  But then by definition of $\mathcal{T}$, we have $Y \in \mathcal{T}$ by condition (2), a contradiction.  Hence we may assume that for every $2$-leaf $v$ in $T$, the diameter of $Y[L[v]]$ is at least $3$.

    Suppose $v \in V(Y)$ is a 2-leaf such that $Y[L[v]]$ has diameter at least 3, and let $Y'' = Y - L(v)$. Observe $v$ is a leaf in $Y''$. By Lemma~\ref{Lemma: k-leaf Reduction}, we have $\gamma_{all, 2}^\infty(Y) = \gamma_{all, 2}^\infty(Y'') + 1$, and by Lemma~\ref{Lemma: Domination >= Eternal 2-Domination}, we have $\gamma_{all, 2}^\infty(Y'') \le \gamma(Y'')$.  Let $w$ and $x$ be the ends of a longest path in $Y[L[v]]$. Then $x$ and $w$ are leaves and $3\leq \dist(w, x) \leq 4$. Then there is a minimum dominating set of $Y$ that contains the unique neighbour of $w$ and the unique neighbour of $x$ both of which are members of $N[v]$.

    Suppose $\dist(w,x)=3$ and assume without loss of generality that $\dist(w, v) = 2$. Note that as $\dist(w,x)=3$ and $\dist(w, v) = 2$, $v$ must be the neighbour of $x$. Let $D$ be such a minimum dominating set of $Y$ containing $v$ (the unique neighbour of $x$). It follows that the neighbour of $w$ is not necessary to dominate any vertex in $Y''$, so $\gamma(Y'') < \gamma(Y)$. As before this implies that $\gamma(Y)-1 = \gamma_{all,2}^\infty(Y)-1 = \gamma_{all,2}^\infty(Y'') \leq \gamma(Y'') \leq \gamma(Y)-1$, so $\gamma(Y'') = \gamma(Y)-1$. Then $D \setminus N(w)$ (which contains $v$) is a minimum dominating set of $Y''$ and $\gamma_{all,2}^\infty(Y'') = \gamma(Y'')$ and by the minimiality of $Y$, $Y'' \in \mathcal{T}$. But this implies $Y \in \mathcal{T}$ by condition (3), a contradiction.

    Suppose then that $\dist(w,x)=4$. Then $\dist(w,v)=\dist(x,v)= 2$. Let $D$ be a minimum dominating set of $Y$ containing the neighbour of $w$ and $x$, call them $u_x,u_w$. As both $u_x,u_w$ are neighbours of $v$, $D\setminus\{u_x,u_w\}$ is a dominating set for $Y''-v$, thus, $\gamma(Y''-v)\leq \gamma(Y) - 2$. Additionally $(D \cup \{v\}) \setminus \{u_x,u_w\}$ is clearly a dominating in $Y''$. Hence, $\gamma(Y'') \leq \gamma(Y)-1$. Again, $\gamma(Y)-1 = \gamma_{all,2}^\infty(Y)-1 = \gamma_{all,2}^\infty(Y'') \leq \gamma(Y'') \leq \gamma(Y)-1$, so we conclude $\gamma(Y'') = \gamma(Y)-1$ and $\gamma_{all,2}^\infty(Y'') = \gamma(Y'')$ implying $Y''\in \mathcal{T}$. Observe that $\gamma(Y'') \leq \gamma(Y''-v)+1$ as any dominating set of $Y''-v$ union $v$ dominated $Y''$. Then $\gamma(Y''-v) = \gamma(Y'')-1$. But this is a contradiction as condition (4) implies $Y \in \mathcal{T}$. This concludes the proof.
\end{proof}

\section{Eternal Distance-$2$ Domination Critical Trees}

We say a graph $G$ is eternal distance-$2$ domination critical if deleting any non-cut vertex, $v$, ensures that $G-v$ has eternal distance-$2$ domination number strictly less than $G$. Of course if $T$ is a tree this is equivalent to stating that deleting any leaf from $T$ reduces the eternal distance-$2$ domination number of $T$. In this section we characterize which trees are eternally distance-$2$ domination critical.  We begin with the following observation.

\begin{lemma}\label{Lemma: Stem-Ciritical}
    If $T = (V,E)$ is a tree and $v \in V$ is a vertex adjacent to two leaves $u,w$, then $\gamma_{all,2}^\infty(T) = \gamma_{all,2}^\infty(T-u)$. 
\end{lemma}

\begin{proof}
    Suppose $T = (V,E)$ is a tree and $v \in V$ is a vertex adjacent to two leaves $u,w$. It is clear that $\gamma_{all,2}^\infty(T-u) \leq \gamma_{all,2}^\infty(T)$ so it is sufficient to show $\gamma_{all,2}^\infty(T-u) \geq \gamma_{all,2}^\infty(T)$. By definition, there exists an initial configuration $D$ in $T-u$ such that for any infinite sequence of attacks $\mathcal{A}'$ there exists a strategy for deploying the guards $\mathcal{D}'$ which is eternally distance-$2$ dominating. Given $\mathcal{D}'$ we define $\mathcal{D}$ to be the same strategy in $T$, except if $u$ is attacked the guards respond as if $w$ were attacked in $T-u$, with the exception of the guard who would move to $w$ who instead moves to $u$. As $\dist(x,u)=\dist(x,w)$ for all $x \in V\setminus \{u,w\}$, then $\mathcal{D}$ is eternally distance-$2$ dominating. Hence, $\gamma_{all,2}^\infty(T-u) \geq \gamma_{all,2}^\infty(T)$ as required.
\end{proof}

Note that Lemma~\ref{Lemma: Stem-Ciritical} implies that if  $T$ is eternal distance-$2$ domination critical, then every stem of $T$ has exactly one leaf. Using this observation we are able to show the following result regarding the structure of $2$-leaves in an eternal distance-$2$ domination critical tree.

\begin{lemma}\label{Lemma: Critical L[v]}
    If $T = (V,E)$ is eternal distance-$2$ domination critical, then for all $2$-leaves $v\in V$, $L[v]$ is isomorphic to $P_3$ or $P_4$.
\end{lemma}

\begin{proof}
Suppose $T = (V,E)$ is an eternal distance-$2$ domination critical tree and let $v \in V$ be a $2$-leaf in $T$. Suppose $v$ is adjacent to multiple 1-leaves $x_1, x_2$ with respective leaves $y_1, y_2$.  By Lemma~\ref{Lemma: k-leaf Reduction}, $\gamma_{all, 2}^\infty(T) = \gamma_{all, 2}^\infty(T - L(v)) + 1$.  Now, for $T - y_2$, we have again by Lemma~\ref{Lemma: k-leaf Reduction} that $\gamma_{all, 2}^\infty(T - y_2) = \gamma_{all, 2}^\infty(T - L(v)) + 1$, since $(T - y_2)[L(v) \setminus\{y_2\}]$ has diameter at least 3.  Thus $\gamma_{all, 2}^\infty(T) = \gamma_{all, 2}^\infty(T - y_2)$, which contradicts $T$ being eternal distance-2 domination critical.  It follows, together with Lemma~\ref{Lemma: Stem-Ciritical}, that $T[L[v]]$ is isomorphic to $P_3$ or $P_4$.
\end{proof}

In the following pair of results, we demonstrate that when these structured 2-leaves exist in an eternal distance-2 domination critical graph, it must have been built from an eternal distance-2 domination critical graph.

\begin{lemma}\label{Lemma: Critical Extension P_3}
    Let $T$ be a tree, and let $T'$ be formed by appending a path on three vertices to a leaf of $T$.  Then $T'$ is eternal distance-2 domination critical if and only if $T$ is eternal distnace-2 domination critical.
\end{lemma}

\begin{proof}
    Let $v$ be a leaf in both $T$ and $T'$.  By Lemma~\ref{Lemma: k-leaf Reduction}, we have $\gamma_{all, 2}^\infty(T') = \gamma_{all, 2}^\infty(T) + 1$ and $\gamma_{all, 2}^\infty(T' - v) = \gamma_{all, 2}^\infty(T - v) + 1$.  Thus $\gamma_{all, 2}^\infty(T - v) < \gamma_{all, 2}^\infty(T)$ if and only if  $\gamma_{all, 2}^\infty(T' - v) < \gamma_{all, 2}^\infty(T')$.

    Now, let $\ell$ be the leaf of the appended path, and let $x$ be the leaf of $T$ to which the path was appended.  By Lemma~\ref{Lemma: k-leaf Reduction}, we have $\gamma_{all, 2}^\infty(T') = \gamma_{all, 2}^\infty(T) + 1$ and $\gamma_{all, 2}^\infty(T' - \ell) = \gamma_{all, 2}^\infty(T - x) + 1$. 
    Thus $\gamma_{all, 2}^\infty(T - x) < \gamma_{all, 2}^\infty(T)$ if and only if $\gamma_{all, 2}^\infty(T' - \ell) < \gamma_{all, 2}^\infty(T')$.  Therefore, $T'$ is eternal 2-domination critical if and only if $T$ is eternal 2-domination critical.
\end{proof}

\begin{lemma}\label{Lemma: Critical Extension P_4}
     Let $T$ be a tree and let $T'$ be formed by appending a path on two vertices and a single vertex to a leaf of $T$.  Then $T'$ is eternal 2-domination critical if and only if $T$ is eternal 2-domination critical.
\end{lemma}

\begin{proof}
Let $T = (V,E)$ be a tree and let $v \in V$ be a leaf in $T$. Form $T'$ by appending a path on two vertices and a single vertex to $v$. Then Lemma~\ref{Lemma: k-leaf Reduction} implies $\gamma_{all,2}^\infty(T') = \gamma_{all,2}^\infty(T'-L(v)) + 1 = \gamma_{all,2}^\infty(T) + 1$ as $T'-L(v) = T$. Let $u$ be the leaf of $T'$ resulting from the path of length $2$ appended to $v$ and let $w$ be the lone vertex appended to $v$.

Let $x$ be a leaf in both $T$ and $T'$.  By Lemma~\ref{Lemma: k-leaf Reduction}, we have $\gamma_{all, 2}^\infty(T') = \gamma_{all, 2}^\infty(T) + 1$ and $\gamma_{all, 2}^\infty(T' - x) = \gamma_{all, 2}^\infty(T - x) + 1$.  Thus $\gamma_{all, 2}^\infty(T - x) < \gamma_{all, 2}^\infty(T)$ if and only if $\gamma_{all, 2}^\infty(T' - x) < \gamma_{all, 2}^\infty(T')$.

By Lemma~\ref{Lemma: k-leaf Reduction}, we have $\gamma_{all, 2}^\infty(T') = \gamma_{all, 2}^\infty(T) + 1$ and $\gamma_{all, 2}^\infty(T' - w) = \gamma_{all, 2}^\infty(T - v) + 1$.  Thus $\gamma_{all, 2}^\infty(T - v) < \gamma_{all, 2}^\infty(T)$ if and only if $\gamma_{all, 2}^\infty(T' - w) < \gamma_{all, 2}^\infty(T')$.

Moreover, it is clear that $\gamma_{all, 2}^\infty(T' - u) \leq \gamma_{all, 2}^\infty(T-v)+1$ given $T'[L[v]\setminus \{u\}]$ has diameter $2$. Recalling $\gamma_{all,2}^\infty(T') = \gamma_{all,2}^\infty(T) + 1$ if $\gamma_{all, 2}^\infty(T - v) < \gamma_{all, 2}^\infty(T)$, then $\gamma_{all, 2}^\infty(T' - u) < \gamma_{all, 2}^\infty(T')$. So if $T$ is eternal distance-$2$ domination critical, then $T'$ is eternal distance-$2$ domination critical.

Suppose now that $T'$ is eternal distance-$2$ domination critical but $T$ is not. Then there exits a leaf $z \in V$ such that $\gamma_{all,2}^\infty(T-z) = \gamma_{all,2}^\infty(T)$. If $z \neq v$, then we have already showed this leads to a contradiction given $z$ will be a leaf in $T$ and $T'$. 

If $z = v$, then $\gamma_{all,2}^\infty(T-v) = \gamma_{all,2}^\infty(T)$. By our assumption that $T'$ is eternal distance-$2$ domination critical $\gamma_{all,2}^\infty(T'-w) = \gamma_{all,2}^\infty(T')-1$ as $T'$ requires more guards than $T'-w$ and placing a guard on $w$ then defending the rest of $T'$ with $\gamma_{all,2}^\infty(T'-w)$ guards is sufficient. Recall Lemma~\ref{Lemma: k-leaf Reduction} implies $\gamma_{all,2}^\infty(T'-w) = \gamma_{all,2}^\infty(T-v)+1$. Hence,
\[
\gamma_{all,2}^\infty(T') = \gamma_{all,2}^\infty(T'-w) + 1 = \gamma_{all,2}^\infty(T-v)+2 = \gamma_{all,2}^\infty(T)+2 
\]
contradicting Lemma~\ref{Lemma: k-leaf Reduction} which implies $\gamma_{all,2}^\infty(T') = \gamma_{all,2}^\infty(T)+1$. Thus, if $T'$ is eternal distance-$2$ domination critical, then $T$ is eternal distance-$2$ domination critical. This concludes the proof.
\end{proof}

The previous results allow us to provide a characterization of eternal distance-2 domination critical graphs.

\begin{theorem} \label{Thm: Critical Trees}
    Let $\mathcal{C}$ be the family of all trees $T$ that can be obtained from the sequence $T_1, \ldots, T_j$ of trees such that $T_1$ is the single vertex tree $K_1$ and $T = T_j$, such that $T_{i + 1}$ is the 1-sum of $P_4$ and $T_i$ at a leaf of $T_i$ and any vertex of $P_4$.  If $T$ is a tree, then $T$ is eternal distance-2 domination critical if and only if $T \in \mathcal{C}$.
\end{theorem}

\begin{proof}
    Suppose $T \in \mathcal{C}$.  If $T = K_1$, then $T$ is trivially eternal distance-2 domination critical.  If $T = P_4$, we have $\gamma_{all, 2}^\infty(P_4) = 2$ and $\gamma_{all, 2}^\infty(P_3) = 1$ by Lemma~\ref{Lemma: Small Diameter Characterization}, so $T$ is eternal distance-2 domination critical.

    Suppose all trees $S \in \mathcal{C}$ on fewer than $n$ vertices are eternal distance-2 domination critical, and let $T \in \mathcal{C}$ be a tree with $n$ vertices.  Then $T$ is the 1-sum of $P_4$ and some $T' \in \mathcal{C}$ at a leaf of $T'$.  It follows by Lemma~\ref{Lemma: Critical Extension P_3} or Lemma~\ref{Lemma: Critical Extension P_4} that $T$ is eternal distance-2 domination critical.

    Conversely, suppose there exists a tree $X$ such that $X \notin \mathcal{C}$ but $X$ is eternally distance-2 domination critical.  Let $Y$ be a minimal such tree with respect to induced subgraph.  If the diameter of $Y$ is at most 2, then by Lemma~\ref{Lemma: Small Diameter Characterization}, $\gamma_{all, 2}^\infty(T) = 1$, so $Y = K_1$, and $Y \in \mathcal{C}$, a contradiction.  If the diameter of $Y$ is 3, then Lemma~\ref{Lemma: Stem-Ciritical} implies $Y = P_4$, so $Y \in \mathcal{C}$, a contradiction.  Hence, we may assume $Y$ has diameter at least 4.  Therefore, by Lemma~\ref{Lemma: Big Diameter Implies k-leaf}, $Y$ has a 2-leaf $v$, and by Lemma~\ref{Lemma: Critical L[v]}, $L[v]$ is isomorphic to $P_3$ or $P_4$.  It follows that $Y$ is the 1-sum of $P_4$ with some tree $Z$.  By Lemma~\ref{Lemma: Critical Extension P_3} and Lemma~\ref{Lemma: Critical Extension P_4}, $Z$ is eternal distance-2 domination critical, so by minimality of $Y$, $Z \in \mathcal{C}$.  Hence, by construction, $Y \in \mathcal{C}$, which is the contradiction completing the proof.
\end{proof}

\section{Trees with $\gamma_{all,2}^\infty = \gamma_2$}

Recall the family $\mathbb{T}$ introduced for Theorem~\ref{lemma: gamma=gamma_2}. 
 Perhaps surprisingly we will show that $\mathbb{T}$ is also the class of graphs with $\gamma_2(T) = \gamma_{all,2}^\infty(T)$. That is there is no tree $T$ such that $\gamma_2(T) = \gamma_{all,2}^\infty< \gamma(T)$. Equivalently $\gamma_{all,2}^\infty(T) = \gamma_2(T)$ if and only if $\gamma(T)= \gamma_2(T)$.

\begin{theorem}\label{Thm: gamma_2 = gamma_(all,2)}
Let  $T$ be a tree. Then $\gamma_{all, 2}^\infty(T) = \gamma_2(T)$ if and only if $T \in \mathbb{T}$.
\end{theorem}

\begin{proof}
If $T\in \mathbb{T}$, then Lemma~\ref{lemma: gamma=gamma_2} implies $\gamma(T) = \gamma_2(T)$. But $\gamma_2(T) \leq \gamma_{all,2}^\infty(T)$ trivially while $\gamma_{all,2}^\infty(T) \leq \gamma(T)$ by Lemma~\ref{Lemma: Domination >= Eternal 2-Domination}. Hence, $\gamma(T) = \gamma_2(T)$ implies $\gamma_2(T) = \gamma_{all,2}^\infty(T)$.

Let $T$ be a tree with $\gamma_{all, 2}^\infty(T) = \gamma_2(T)$.  Let $v_0, v_1, \ldots, v_k$ be a longest path in $T$.  If $k \le 2$, then $T$ is $P_1$ or a star $K_{1, m}$ for some non-negative integer $m$, and clearly $T$ is in $\mathbb{T}$.

If $k \in \{3, 4\}$, then $\gamma_2(T) = 1$, but $\gamma_{all, 2}^\infty(T) > 1$.  Hence, we may assume $k \ge 5$.  We proceed by induction on the number $n(T)$ of vertices of a tree $T$ with $\gamma_{all, 2}^\infty(T) = \gamma_2(T)$. It is easy to check that there is no graph on strictly less than $6$ vertices with $k\geq 5$ and $\gamma_{all, 2}^\infty(T) = \gamma_2(T)$. Suppose then that $n(T) \geq 6$. If $n(T) = 6$, then $T \cong P_6$ is the only graph with $\gamma_{all, 2}^\infty(T) = \gamma_2(T)$ and $P_6 \in \mathbb{T}$, by applying operation $\mathbb{T}_3$ to $P_2$.  Now let $T$ be a tree with $\gamma_{all, 2}^\infty(T) = \gamma_2(T)$ and $n(T) \ge 7$, and assume that each tree $T'$ with $n(T') < n(T)$, $k \ge 5$, and $\gamma_{all, 2}^\infty(T') = \gamma_2(T')$ is in $\mathbb{T}$.

Suppose $T$ contains a stem $v$ with leaves $x$ and $y$.  By Lemma~\ref{Lemma: Stem-Ciritical}, we have $\gamma_{all, 2}^\infty(T) = \gamma_{all, 2}^\infty(T - x)$.  It is clear that any 2-dominating set of $T - x$ also dominates $T$, since the guard that dominates $y$ is at most distance two from both $x$ and $y$.  Thus $\gamma_2(T) = \gamma_2(T - x)$.  Therefore, $\gamma_{all, 2}^{\infty}(T - x) = \gamma_2(T - x)$, so $T - x \in \mathbb{T}$.  But then $T \in \mathbb{T}$ by Operation $\mathbb{T}_1$, a contradiction.  Thus, we may assume every stem of $T$ has exactly one leaf.

As $v_0$ must be a leaf, this implies $\deg(v_1) = 2$.  Suppose $\deg(v_2) > 2$.  Then $v_2$ is adjacent to a stem or leaf $x$.  Let $D$ be the configuration of a minimum set of guards to eternally distance-2 dominate $T$ after an attack on $x$.  Then in order to defend against a possible attack at $v_0$, at least one of $v_0, v_1, v_2 \in D$.  Let $D' = (D - L[v_2]) \cup \{v_2\}$.  Then $|D'| < |D|$ and $D'$ distance 2-dominates $T$, contradicting that $\gamma_{all, 2}^\infty(T) = \gamma_2(T)$.  Thus $\deg(v_2) = 2$.

Suppose $\deg(v_3) > 2$.  If $v_3$ is a stem with leaf $x$, then let $D$ be the configuration of a minimum set of guards to eternally distance-2 dominate $T$ after an attack on $x$.  Then in order to defend against a possible attack at $v_0$, at least one of $v_0, v_1, v_2 \in D$.  Let $D' = (D - L[v_2] - \{x\}) \cup \{v_2\}$.  Then $|D'| < |D|$ and $D'$ distance-2 dominates $T$, contradicting that $\gamma_{all, 2}^\infty(T) = \gamma_2(T)$.  Thus, $v_3$ is not a stem of $T$.  Let $T' = T - L[v_2]$.  Since $\deg_T(v_3) > 2$, $v_3$ is not a leaf in $T'$, and since $k \ge 5$, $v_3$ is not a stem in $T'$.  By Lemma~\ref{Lemma: k-leaf Reduction}, we have $\gamma_{all, 2}^\infty(T) = \gamma_{all, 2}^\infty(T') + 1$, and $\gamma_2(T) \le \gamma_2(T') + 1$.  Therefore,  Lemma~\ref{Lemma: Domination >= Eternal 2-Domination} implies,
\[
\gamma_2(T) \le \gamma_2(T') + 1 \le \gamma_{all, 2}^\infty(T') + 1 = \gamma_{all, 2}^\infty(T) = \gamma_2(T).
\]
Thus, $\gamma_2(T') = \gamma_{all, 2}^\infty(T')$ so by the induction hypothesis, $T' \in \mathbb{T}$.  Since $T$ is obtained from $T'$ by operation $\mathbb{T}_2$, we conclude $T \in \mathbb{T}$.

Thus, we assume $\deg(v_3) = 2$.  Let $D$ be a configuration of a minimum set of guards to eternally distance-2 dominate $T$ after an attack on $v_3$.  Then in order to defend against a possible attack at $v_0$, at least one of $v_0, v_1, v_2 \in D$.  Let $D' = (D - L[v_3]) \cup \{v_2\}$.  Since $|D'| < |D|$, we have that $D'$ is not a distance-2 dominating set of $T$.  Thus $v_3 \in D$ and $v_4$ has a neighbour, say $u$, that is not distance-2 dominated by $D'$. 

Consider the components of $T - u$.  Each component can be eternally distance-2 dominated by its vertices in $D$, since no vertex of $D$ except $v_3$ is within distance 
$2$ of $u$. If $\deg(u) > 1$, then let $z \neq v_4$ be a neighbour of $u$, and let $\hat{D}$ be the configuration of the guards after an attack on $z$, starting from the configuration $D$. It is clear that the guard which defends the attack at $z$ must be in a component of $T-u$ distinct from $v_4$, as otherwise $v_4 \in D$ which would imply that $u$ be within distance $2$ of a vertex in $D'$. Let $\hat{D}_u = \hat{D}$. 

Then for all non-leaf neighbours of $v_4$ that are not distance-$2$ dominated in $D'$, $u_1,\dots, u_l$, we can define sets $\hat{D}_{u_i}$ as above. Furthermore, as each attack on a vertex $z_i = z$ (given $u_i = u$) calls for a distinct set of guards $g_i$ to defend against it, there is a eternally distance-$2$ dominating set $\mathcal{D}$ given by all guards $g_i$ defending against attacks at $z_i$, for all $1 \leq i \leq l$, while the rest of the guards in $D$ do not move. Then $\mathcal{D}' = (\mathcal{D} - L[v_3]) \cup \{v_2\}$ is a distance-2 dominating set of $T$, given all vertices of the $2$-neighbourhood of $v_3$ are dominated, and $|\mathcal{D}'| < |\mathcal{D}|$,  contradicting that $\gamma_{all, 2}^\infty(T) = \gamma_2(T)$.  Hence all neighbours of $v_4$ are leaves in $T$, and therefore $v_4$ is a stem. Let $T' = T - L[v_3]$. We can verify that $\gamma_{all, 2}^\infty(T') + 1 = \gamma_{all, 2}^\infty(T)$ and $\gamma_2(T') + 1 = \gamma_2(T)$.  Thus, $\gamma_{all, 2}^\infty(T') = \gamma_2(T')$, so by the induction hypothesis, $T' \in \mathbb{T}$.  Since $T$ is obtained from $T'$ by operation $\mathbb{T}_3$, we conclude $T \in \mathbb{T}$. 
\end{proof}

\section{Extremal Families of Trees for Eternal Distance-$k$ Domination}

Motivated by a question posed in \cite{cox2023eternal}, we explore the extreme values the eternal distance-$k$ domination number can take in trees. We begin by giving a general upper bound for the eternal distance-$k$ domination number of trees. As the eternal distance-$k$ domination number of a graph is upper bounded by the eternal distance-$k$ domination number of its spanning subgraphs, this implies a general upper bound for connected graphs. Note this generalises a result of Chambers, Kinnersley, and Prince \cite{chambers2006mobile}. Beyond this we construct families of trees which meet this upper bound, as well as a family of trees which has eternal domination number equal to $\frac{n}{1+ \sum_{i=1}^k\Delta(\Delta-1)^{i-1}}$, which is the lowest possible value the distance-$k$ domination number can take.

\begin{theorem}\label{Thm: General Upper Bound}
If $G$ is a connected graph of order $n$, then $\frac{n}{1+ \sum_{i=1}^k\Delta(\Delta-1)^{i-1}} \leq \gamma_{all,k}^\infty(G) \leq \left \lceil \frac{n}{k+1} \right \rceil$.
\end{theorem}

\begin{proof}
The lower bound is trivial as each vertex can distance-$k$ dominate at most $1+ \sum_{i=1}^k\Delta(\Delta-1)^{i-1}$ vertices. So $\gamma_{all,k}^\infty(G) \geq \gamma_k(G) \geq \frac{n}{1+ \sum_{i=1}^k\Delta(\Delta-1)^{i-1}}$. It remains to be shown that $\gamma_{all,k}^\infty(G) \leq \left \lceil \frac{n}{k+1} \right \rceil$.

As deleting edges will never help the guards defend the graph it is sufficient to prove the result for trees. Suppose then that $G = T = (V,E)$ is a tree. When the diameter of $T$ is strictly less than $2k$ the result follows by Lemma~\ref{Lemma: Small Diameter Characterization}. Otherwise, the diameter of $T$ is at least $2k$, in which case Lemma~\ref{Lemma: Big Diameter Implies k-leaf} implies that there exists a $k$-leaf $v \in V$. 

Suppose $T$ is a smallest counterexample. Then the diameter of $T$ is at least $2k$. Let $v \in V$ be a $k$-leaf in $T$ and let $T'=T - L[v]$ when the diameter of $T[L[v]]$ is $k$, and $T'=T - L(v)$ otherwise. By the definition of a $k$-leaf this implies $|V(T')|\leq n-(k+1)$. By the minimality of $T$, $T'$ is not a counterexample, hence, 
\[
\gamma_{all,k}^\infty(T') \leq \left \lceil \frac{n-(k+1)}{k+1} \right \rceil = \left \lceil \frac{n}{k+1} \right \rceil - 1.
\]
Let $\gamma_{all,k}^\infty(T')$ guards defend the subgraph $T'$ of $T$ and place a single guard on $v$. Note that this will be at most $\left \lceil \frac{n}{k+1} \right \rceil$ guards. If the diameter $T[L[v]]$ is $k$, then the guard at $v$ can protect $L[v]$ indefinitely while the $\gamma_{all,k}^\infty(T')$ guards protect $T'$ indefinitely. Otherwise, the diameter $T[L[v]]$ is at least $k+1$. In this case, let the guards proceed by the same strategy described in case 2 of the proof of Lemma~\ref{Lemma: k-leaf Reduction}. This strategy will protect $T$ indefinitely.

Thus, 
\[
\gamma_{all,k}^\infty(T) \leq \gamma_{all,k}^\infty(T') +1 \leq \left \lceil \frac{n}{k+1} \right \rceil,
\]
which concludes the proof.
\end{proof}

Note that \cite{cox2023eternal} showed that paths are a family of graphs that meet the upper bound of Theorem~\ref{Thm: General Upper Bound}. The reason for this is because in any distance-$k$ dominating set we must maintain guards within distance $k$ of each leaf, therefore to protect against attacks on leaves these guards can never leave the $k$ neighbourhood of each leaf. Hence, more guards are required to protect vertices which are distance slightly more than $k$ from a leaf. By induction this will force paths to require the maximum possible number of guards.

Using a similar ideas we construct a large family $\mathcal{T}_{M,k}$ of trees which also have eternal distance-$k$ domination number  $\left \lceil \frac{n}{k+1} \right \rceil$. We define $\mathcal{T}_{M,k}$ as follows. Let $T'$ be any tree, construct $T \in \mathcal{T}_{M,k}$ by connecting a path $P_v = v_1 \dots v_{k}$ to each vertex $v \in V(T')$. That is for each $v$ add an edge $(v,v_1)$ so that $P_v$ becomes part of $T$.

\begin{figure}
    \centering
    \includegraphics[scale = 0.5]{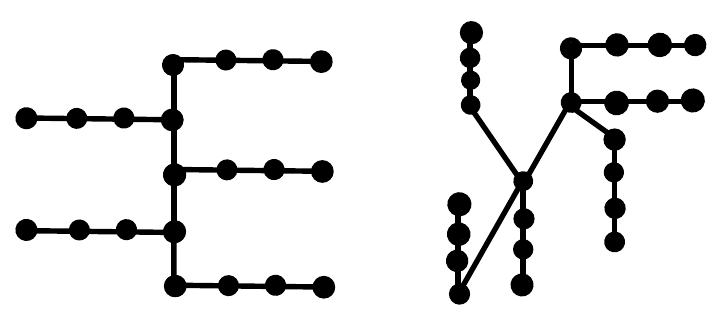}
    \caption{Two examples of trees in $\mathcal{T}_{M,3}$.}
    \label{fig:T_{M,k} example}
\end{figure}

\begin{theorem}
If $T \in \mathcal{T}_{M,k}$, then $\gamma_{all,k}^\infty(T) = \left \lceil \frac{n}{k+1} \right \rceil$.
\end{theorem}

\begin{proof}
Let $T = (V,E) \in \mathcal{T}_{M,k}$, then $T$ is given by appending paths of length $k$ to each vertex in some other tree $T'$. Notice that $|V| = n =(k+1)|V(T')|$. We aim to show that $\gamma_{all,k}^\infty(T) = |V(T')|$. By Theorem~\ref{Thm: General Upper Bound}, $\gamma_{all,k}^\infty(T) \leq |V(T')|$ so all that remains to be  shown is that $\gamma_{all,k}^\infty(T) \geq |V(T')|$.

Note that in each distance-$k$ dominating set of $T$ for each path $P_v$ in $T$ appended to a vertex $v \in V(T')$ there must be at least $1$ guard in $P_v\cap \{v\}$ to distance-$k$ dominate the leaf at the end of $P_v$. As for each $u\neq v \in V(T')$, $(P_v \cup \{v\})\cap (P_u \cup \{u\}) = \emptyset$ this implies $\gamma_k(T) \geq |V(T')|$. Hence, $\gamma_{all,k}^\infty(T) \geq \gamma_k(T) \geq |V(T')|$ as required. This completes the proof.
\end{proof}

Now we construct a family of trees $\mathcal{T}_{m,k, \Delta}$ which all have eternal distance-$k$ domination number exactly $\frac{n}{1+ \sum_{i=1}^k\Delta(\Delta-1)^{i-1}}$ for odd $k$. Let $T_{k,\Delta}$ be the complete $(\Delta-1)$-ary tree of depth $\left \lfloor \frac{k}{2} \right \rfloor$ except the root vertex has $\Delta$ followers rather than $\Delta-1$. Let $T_{k,\Delta} \in \mathcal{T}_{m,k, \Delta}$, we complete out definition of $\mathcal{T}_{m,k, \Delta}$ as follows; for any $T_1,T_2 \in \mathcal{T}_{m,k, \Delta}$ let $T_3$ be any tree formed by adding an edge between two vertices of $T_1$ and $T_2$ which have degree strictly less than $\Delta$, then $T_3 \in \mathcal{T}_{m,k, \Delta}$. Notice that like $\mathcal{T}_{M,k}$, $\mathcal{T}_{m,k,\Delta}$ contains arbitrarily many vertices of high degree.

\begin{figure}
    \centering
    \includegraphics[scale = 0.4]{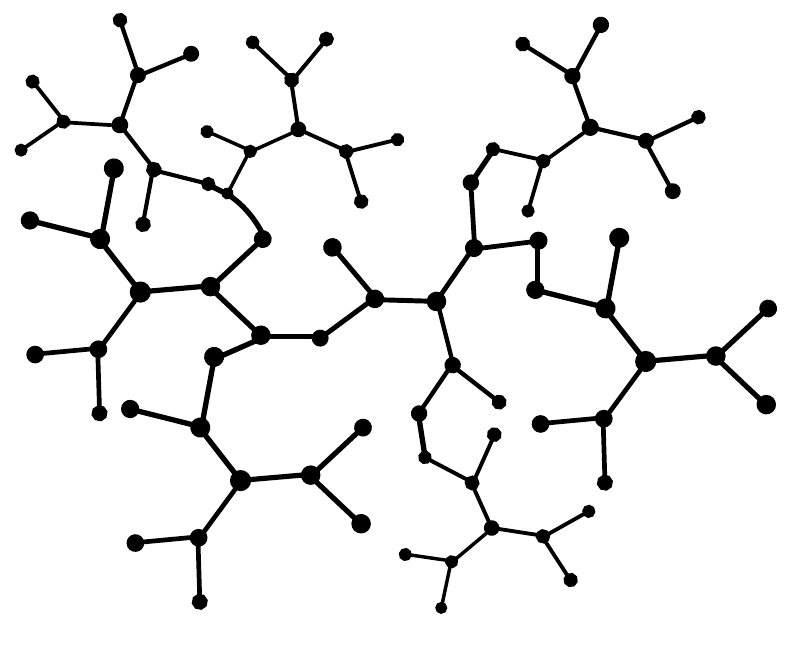}
    \caption{A tree in $\mathcal{T}_{m,5,3}$.}
    \label{fig:T_{m,k,D} example}
\end{figure}

\begin{theorem}
If $T \in \mathcal{T}_{m,k,\Delta}$, then $\gamma_{all,k}^\infty(T) = \frac{n}{1+ \sum_{i=1}^k\Delta(\Delta-1)^{i-1}}$.
\end{theorem}

\begin{proof}
Let $T = (V,E) \in \mathcal{T}_{m,k,\Delta}$. Theorem~\ref{Thm: General Upper Bound} implies that $\gamma_{all,k}^\infty(T) \geq \frac{n}{1+ \sum_{i=1}^k\Delta(\Delta-1)^{i-1}}$ so it is sufficient to show $\gamma_{all,k}^\infty(T) \leq \frac{n}{1+ \sum_{i=1}^k\Delta(\Delta-1)^{i-1}}$. Note that if $T = T_{k,\Delta}$, then the result follows from Lemma~\ref{Lemma: Small Diameter Characterization}. Suppose then that $T \neq T_{k,\Delta}$ is a smallest counterexample.

By the definition of $\mathcal{T}_{m,k,\Delta}$, there exists an edge $e \in E$ such that $T-e$ is a graph with two connected components $T_1$ and $T_2$ each of which is $\mathcal{T}_{m,k,\Delta}$. As $T$ is a smallest counterexample, $T_1$ and $T_2$ are not counterexamples. Thus, if $|V(T_1)| = a$ and $|V(T_1)| = b$, we know $n = a+b$ and $\gamma_{all,k}^{\infty}(T_1) = \frac{a}{1+ \sum_{i=1}^k\Delta(\Delta-1)^{i-1}}$ and $\gamma_{all,k}^{\infty}(T_2) = \frac{b}{1+ \sum_{i=1}^k\Delta(\Delta-1)^{i-1}}$. 

It is not hard to see that if as $V = V(T_1)\cup V(T_2)$ and $T_1$, $T_2$ are subgraphs of $T$, 
\begin{align*}
    \gamma_{all,k}^{\infty}(T) \leq \gamma_{all,k}^{\infty}(T_1) + \gamma_{all,k}^{\infty}(T_2) 
    = \frac{a}{1+ \sum_{i=1}^k\Delta(\Delta-1)^{i-1}}+\frac{b}{1+ \sum_{i=1}^k\Delta(\Delta-1)^{i-1}} \\
    = \frac{n}{1+ \sum_{i=1}^k\Delta(\Delta-1)^{i-1}}
\end{align*}
this completes the proof.
\end{proof}

\section{Conclusion}

We have resolved a number of questions regarding the eternal distance-$2$ domination number of trees, in particular several questions raised in \cite{cox2023eternal}. We do this by giving a polynomial time algorithm (Theorem~\ref{Thm: complexity}) for calculating the eternal distance-$2$ domination of a tree, then using the reductions involved in this algorithm we characterize which trees have domination number equal to their eternal distance-$2$ domination number (Theorem~\ref{Thm: gamma = gamma_(all,2)}). Beyond this we characterize which trees are eternal distance-$2$ domination critical (Theorem~\ref{Thm: Critical Trees}) and which trees have eternal distance-$2$ domination number equal to their distance-$2$ domination number (Theorem~\ref{Thm: gamma_2 = gamma_(all,2)}). These results are similar to work by Klostermeyer and MacGillivray \cite{klostermeyer2014eternal} who proved several characterizations for trees in the context of eternal domination number.

Additionally, we generalize upper bounds on the eternal domination number given by Chambers, Kinnersley, and Prince  \cite{chambers2006mobile} to the eternal distance-$k$ domination number. We also give a lower bound for the eternal distance-$k$ domination number in terms of order and maximum degree. To demonstrate that both of these bounds are tight we construct infinite families of tree that meet each bound. We conclude the paper by listing several open problems and conjectures.

The first of these conjectures arise from our observation that eternal distance-$k$ domination seems to have some fundamental differences in the $k \leq 2$ and $k>2$ cases. This is highlighted by the effectiveness of Lemma~\ref{Lemma: k-leaf Reduction} for $k \leq 2$ and subsequent failure for $k >2$. Despite this we conjecture the following.

\begin{conjecture}
    For all integers $k>2$, determining $\gamma_{all,k}^\infty(T)$ where $T$ is a tree is polynomial time bounded by a polynomial $p(n)$ independent of $k$.
\end{conjecture}

Next we observe that Theorem~\ref{Thm: gamma_2 = gamma_(all,2)} combined with work in \cite{raczek2011graphs} implies that if $\gamma_2(T) = \gamma_{all,2}^\infty(T)$, then $\gamma_{all,2}^\infty(T) = \gamma(T)$. We conjecture that a similar result holds for $k>2$. Recall that $\gamma_k \leq \gamma_{all,k}^\infty \leq \gamma_{\left \lfloor \frac{k}{2} \right \rfloor}$ for all graphs.

\begin{conjecture}
For all $k>2$ and trees $T$, if $\gamma_k(T) = \gamma_{all,k}^\infty(T)$, then $ \gamma_{all,k}^\infty(T) = \gamma_{\left \lfloor \frac{k}{2} \right \rfloor} (T)$.
\end{conjecture}

The following are several questions which merit study in future work as they are natural to consider yet outside the scope of this paper.

\begin{question}
What is the maximum $m$ as a function of $n$, so that there exists a graph on $n$ vertices and $m$ edges with eternal distance-$k$ domination number $\left \lceil \frac{n}{k+1} \right \rceil$.
\end{question}

\begin{question}
Is $\gamma_{all,k}^\infty(T) \leq  \frac{n}{t}$  for some $t>k+1$ for all graphs with treewidth at least $N$ for some sufficiently large constant $N$?
\end{question}

\begin{question}
Does there exist a family of graphs $\mathcal{G}$ where determining $\gamma_{all,k}^\infty(G)$ is polynomial time but determining $\gamma_{all,t}^\infty(G)$ is not polynomial time for $G \in \mathcal{G}$ and some $k \neq t$. If so which families exhibit this dichotomy.
\end{question}

\vspace{0.5cm}
\section*{Acknowledgements}
We would like to acknowledge the support of the Natural Sciences and Engineering Research Council of Canada (NSERC) through the Canadian Graduate Scholarship -- Master's program.

\bibliographystyle{plain}

\end{document}